\documentclass[12pt,letterpaper]{article}
\usepackage[hmargin=1in,vmargin=1in]{geometry}

\usepackage{amsmath}
\usepackage{amsfonts}
\usepackage{amsthm}
\usepackage{amssymb}
\usepackage{array}
\usepackage{caption}
\usepackage{color}
\usepackage{graphics}
\usepackage{graphicx}
\usepackage{hyperref}
\usepackage{verbatim} 

\newtheorem{prop}{Proposition}
\newtheorem{cor}[prop]{Corollary}
\newtheorem{theorem}[prop]{Theorem}
\newtheorem{lemma}[prop]{Lemma}
\newtheorem{conjecture}[prop]{Conjecture}

\begin{document}

\title{Free Fibonacci Sequences}
\author{Brandon Avila\\ MIT \and Tanya Khovanova\\ MIT}
\maketitle

\begin{abstract}
This paper describes a class of sequences that are in many ways similar to Fibonacci sequences: given $n$, sum the previous two terms and divide them by the largest possible power of $n$. The behavior of such sequences depends on $n$. We analyze these sequences for small $n$: 2, 3, 4, and 5. Surprisingly these behaviors are very different. We also talk about any $n$. Many statements about these sequences are difficult or impossible to prove, but they can be supported by probabilistic arguments, we have plenty of those in this paper.

We also introduce ten new sequences. Most of the new sequences are also related to Fibonacci numbers proper, not just free Fibonacci numbers. 
\end{abstract}

\section{Introduction}
\label{sec:intro}

John Horton Conway likes playing with the Fibonacci sequence. Instead of summing the two previous terms, he sums them up and then adds a twist: some additional operation. John Conway discussed these sequences with the second author and this is how we got interested in them. The second author already wrote about one class of such sequences called subprime Fibonacci sequences jointly with Richard Guy and Julian Salazar \cite{GKS}. Here we discuss another variation called $n$-free Fibonacci sequences.

An $n$-free Fibonacci sequence starts with any two integers: $a_1$ and $a_2$. After that it is defined by the recurrence $a_k = (a_{k-1} +a_{k-2})/n^i$, where $n^i$ is the largest power of $n$ that is a factor of $a_{k-1} +a_{k-2}$. 

It appears that many other people like twisting Fibonacci sequences. After we started working on this paper and made our calculations we checked, as everyone should, the On-Line Encyclopedia of Integer Sequences (OEIS \cite{OEIS}) and discovered that some $n$-free Fibonacci sequences were already submitted by three other people. Surprisingly, the first sequence submitted was the sequence of 7-free Fibonacci numbers (A078414) entered by Yasutoshi Kohmoto in Dec 2002. After that the sequence of 5-free Fibonacci numbers (A214684) was submitted by John W. Layman in Jul 2012. It followed by the sequence of 4-free Fibonacci numbers (A224382) submitted by Vladimir Shevelev in Apr 2013. As the reader will see very soon 2-free and 3-free Fibonacci numbers do not constitute new sequences. We filled the gap and submitted 6-free Fibonacci numbers (A232666) in Nov 2013. 

In Section~\ref{sec:definitions} we introduce useful facts about Fibonacci numbers. In Section~\ref{sec:2free} we show that all 2-free sequences end in a cycle of length 1. The 3-free Fibonacci sequences are much more complicated and we study them in Section~\ref{sec:3free}. All our computational experiments ended in a cycle of length 3. On the other hand, we show that 3-free sequences may contain arbitrary long increasing substrings. We prove this in Section~\ref{sec:customized}. Nevertheless, we give a probabilistic argument that a 3-free sequence should end in a cycle in Section~\ref{sec:3free}.

The 4-free Fibonacci sequences are vastly different from 2-free and 3-free sequences (Section~\ref{sec:4free}). We did not find a sequence that ends in a cycle: all of them grow in our experiments. The proof that all of them grow seems intractable, but we supply a probabilistic argument that this is the case. Yet 5-free sequence bring something new (see Section~\ref{sec:5free}). They contain sequences that are never divided by 5 and provably grow indefinitely. At the same time 5-free sequences contain cycles too.

We continue with Section~\ref{sec:division-free} where we find other numbers $n$ that provide examples of sequences that never need to be divided by $n$. Now we wonder where the cycles disappeared to and discuss there potential properties in Section~\ref{sec:cycles}. 

We finish with a discussion of our computational results. Section~\ref{sec:growthdivfree} explains why the average growth for some $n$-free sequences is close to the golden ratio and Section~\ref{sec:growthomni} explains the growth behavior for other values of $n$.

\section{Fibonacci Numbers and $n$-free Fibonacci Sequences}\label{sec:definitions}

Let us denote \textit{Fibonacci numbers} by $F_k$. We assume that $F_0=0$ and $F_1=1$. The sequence is defined by the Fibonacci recurrence: $F_{n+1} =F_n + F_{n-1}$ (See A000045). We call an integer sequence $a_n$ \textit{Fibonacci-like} if it satisfies the Fibonacci recurrence: $a_k = a_{k-1} +a_{k-2}$. A Fibonacci-like sequence is similar to the Fibonacci sequence, except it starts with any two integers. The second-famous Fibonacci-like sequence is the sequence of \textit{Lucas numbers} $L_i$ that starts with $L_0=2$ and $L_1=1$: 2, 1, 3, 4, 7, 11, $\ldots$ (See A000032).

An \textit{$n$-free Fibonacci} sequence starts with any two integers: $a_1$ and $a_2$ and is defined by the recurrence $a_k = (a_{k-1} +a_{k-2})/n^i$, where $n^i$ is the largest power of $n$ that is a factor of $a_{k-1} +a_{k-2}$. To continue the tradition we call numbers in the $n$-free Fibonacci sequence that starts with $a_0=0$ and $a_1=1$ \textit{$n$-free Fibonacci numbers}.

In the future we will consider only sequences starting with two non-negative integers. It is not that we do not care about other starting pairs, but positive sequences cover all essential cases. Indeed, if we start with two negative numbers we can multiply the sequence by $-1$ and get an all-positive sequence. If we start with numbers of different signs, the sequence eventually will become the same-sign sequence.

If we start with two zeros, we get an all-zero sequence. So we will consider only sequences that do not have two zeros at the beginning. Note, that a non-negative sequence can have a zero only in one of the two starting positions, never later.

The $n$-free Fibonacci sequence coincides with the Fibonacci-like sequence with the same beginning until the first occurrence of a multiple of $n$ in the Fibonacci-like sequence. 

Given a positive integer $m > 1$, the smallest positive index $k$ for which $n$ divides the $k$-th Fibonacci number $F_k$ is called the \textit{entry point} of $m$ and is denoted by $Z(m)$ (see sequence A001177 of Fibonacci entry points). For example, $Z(10) = 15$ and the 10-free Fibonacci numbers coincide with the Fibonacci numbers for indices $ < 15$.

Now that all the preparation is done, let us take a closer look at the simplest $n$-free Fibonacci sequences: 2-free Fibonacci sequences.

\section{2-free Fibonacci Sequences}\label{sec:2free}

Consider some examples. The sequence that starts with 0, 1 continues as 1, 1, 1, .... The only two 2-free Fibonacci numbers are 0 and 1. The sequence eventually stabilizes, or in other words, turns into a cycle of length 1. Let us look at other starting points. The sequence that starts as 1, 2 continues as 3, 5, 1, 3, 1, 1, and stabilizes at 1. This sequence turns into the same cycle. The sequence that starts as 100, 220, continues as, 5, 225, 115, 85, 25, 55, 5, 15, 5, 5, 5, and stabilizes at 5. It turns into a different cycle, but the length of the cycle is again equal to 1.

\begin{lemma}
Every 2-free Fibonacci sequence eventually turns into a cycle of length 1: $x$, $x$, $x$, $\ldots$, for an odd $x$.
\end{lemma}

\begin{proof}
It is clear that after the second term all elements of the sequence are odd. Consider the maximum of the two consecutive terms of the sequence: $m_k = \max\{a_k,a_{k-1}\}$. If two consecutive terms $a_{k-1}$, $a_k$ of the sequence are odd and not equal to each other, then the maximum decreases: $m_{k+1} < m_k$. Hence, the sequence has to stabilize.
\end{proof}

It follows from the proof that for a sequence starting with $a_1$, $a_2$ the number of steps until the cycle is reached is not more than $\max\{a_1,a_2\}$. On the other hand, the subsequence before the cycle can be arbitrary long. It follows from the following lemma.

\begin{lemma}
For any two odd numbers $a_1$, $a_2$, a preceding odd number $a_0$ can be found so that $a_0$, $a_1$, and $a_2$ form a 2-free Fibonacci sequence.
\end{lemma}

\begin{proof}
Pick a positive integer $k$ so that $2^k a_2 > a_1$ and set $a_0$ to be equal to $2^k a_2 - a_1$.
\end{proof}

There are many ways to build predecessors to a given 2-free Fibonacci sequence. The minimal such sequence is built when we choose the smallest power of 2 that still allows us to have positive members in the sequence. We explicitly build such an example starting with $a_1=3$, and $a_2=1$. Reversing the indexing direction we get: 1, 3, 1, 5, 3, 7, 5, 9, 1, $\ldots$, which is now sequence A233526.

Next, we want to continue with 3-free Fibonacci sequences. Are they as simple as 2-free sequences?

\section{3-free Fibonacci Sequences}\label{sec:3free}

Let us look at 3-free Fibonacci sequences. Consider an example of 3-free Fibonacci  numbers: 0, 1, 1, 2, 1, 1, 2, and so on. The sequence turns into a cycle of length 3. There are only 3 different 3-free Fibonacci numbers.

We can multiply a 3-free sequence by a number not divisible by 3 to get another 3-free sequence. Thus, in general we can get cycles of the form $k$, $k$, $2k$, where $k$ is not divisible by 3.

\begin{lemma}
Any cycle of length 3 in a 3-free Fibonacci sequence is of the form $k$, $k$, $2k$.
\end{lemma}

\begin{proof}
Consider the length 3 cycle $a$, $b$, $c$. From the definition of 3-free Fibonacci sequences, we know the following relations:

\begin{eqnarray}
a+b=3^xc \nonumber \\
b+c=3^ya \label{eq:b+c}\\
c+a=3^zb. \label{eq:c+a}
\end{eqnarray}

Furthermore, no term in the sequence is divisible by 3. Then, by the pigeon-hole principle, at least two of the terms $a,b,c$ must be congruent modulo 3. Without loss of generality, take $a \equiv b$ (mod 3). Then $a+b \not\equiv 0$ (mod 3), so we have that $x=0$ and $a+b=c$. Now substitute for $c$ and add equations~(\ref{eq:b+c}) and (\ref{eq:c+a}) to get that $a+b=3^{y-1}a+3^{z-1}b$. Since $3\nmid a+b$, either $y=1$ or $z=1$. If $y=1$, then $b=3^{z-1}b$, hence $z=1$. Similarly, $z=1$ implies $y=1$. In either case, $y=z=1$. Then we may solve for our initial variables to show that $a=b$ and $c=a+b$. Restated, $a=k$, $b=k$, and $c=2k$.
\end{proof}

\begin{cor}
The number $k$ in the cycle is the greatest common divisor of the sequence.
\end{cor}

\begin{proof}
Because of the Fibonacci additive property, if any number divides two or more elements of the sequence (excluding the first two, which may be divisible by 3), it must divide all numbers in the sequence. Thus, $k$ must divide every element. The least of these elements, then, can only be $k$ itself, making it the greatest common divisor.
\end{proof}

Will it be the case that all 3-free Fibonacci sequences end in cycles of length 3? We will build suspense by delaying this discussion, meanwhile we have a lemma about the length of any potential cycle:

\begin{lemma}\label{thm:parity}
Any cycle in a 3-free Fibonacci sequence is of length $3n$ for some integer $n$.
\end{lemma}

\begin{proof}
Begin with any 3-free Fibonacci sequence, and divide out the highest power of 2 in the GCD of all its elements. The resulting sequence is a 3-free Fibonacci sequence with at least one odd element. It is clear that dividing or multiplying any number by 3 does not change its parity. Thus, any sequence, regardless of how many factors of 3 are divided out from each term, will have the same underlying structure in its parity. Since we have reduced the sequence to the point where there exists at least one odd number, we know that the sequence reduced modulo 2 must be congruent to 1, 1, 0, 1, 1, 0. Only cycles of length $3n$ are permitted in this structure, and therefore permitted in 3-free sequences.
\end{proof}

We checked all the starting pairs of numbers from 1 to 1000, and all these sequences end in a 3-cycle.

The 3-free Fibonacci sequences are in many ways similar to the notorious Collatz sequences \cite{La10}, for which it is still not known if every sequence eventually cycles. We do not expect that it is easy to prove or disprove that ever 3-free Fibonacci sequence ends in a cycle. On the other hand, it is possible to make probabilistic arguments to support different claims about free Fibonacci sequences.

Here is the base of the argument. Suppose we encounter a number in a sequence that is divisible by 3, then after removing all powers of 3, let us assume that the resulting number has the remainder 1 modulo 3 with probability $1/2$. If the sum of two consecutive terms in a sequence is large and divisible by 3, then we also assume that this number is divisible by $3^k$ with probability $1/3^{k-1}$.

How often do we divide by 3 in a 3-free Fibonacci sequence? The following lemma is obvious.

\begin{lemma}
In a 3-free Fibonacci sequence the division happens for every term or for every other term.
\end{lemma}

In other words, we can not have a subsequence of length 3 such that each term is the sum of the previous two terms. We want to study two polar cases first: stretches where we divide every term and stretches were we divide every other term.

We will call a subsequence of a 3-free Fibonacci sequence where we divide at each step a \textit{division-rich subsequence}. Conversely, we will call a subsequence of a 3-free Fibonacci sequence where we divide at every other step a \textit{division-poor subsequence}. 

\begin{lemma}\label{thm:rich}
There exist arbitrary long division-rich subsequences.
\end{lemma}

\begin{proof}
The proof is done by explicit construction. Consider the definition of a division-rich subsequence. In this case, we divide by a power of 3 after every addition step, so that $3^{i_n} \cdot a_n = a_{n-1}+a_{n-2}$ for $i_n>0$. Equivalently, $a_{n-2}=3^{i_n} \cdot a_n - a_{n-1}$. Thus, by choosing $a_n$ and $a_{n-1}$, and selecting a sequence $\{i_m\}$ that satisfies this relationship, the sequence can easily be constructed backwards. Our only requirements are that every term of the sequence is positive, and every step contains a division, so it will suffice to construct a sequence  $\{i_m\}$ such that $3^{i_n} \cdot a_n - a_{n-1}>0$ and $i_n>0$ for all $n$.
\end{proof}

As an example, let us begin with $a_n=1$, $a_{n-1}=1$. At each step let us choose the smallest possible power for $i_m$. Then $i_n = 1$ satisfies our inequality for the first step, and $a_{n-2}=2$. Continuing, $i_{n-1}=1$ satisfies the inequality for the next step, yielding $a_{n-3}=1$. Next, $i_{n-2}=1$ and $a_{n-3}=1$, followed by $i_{n-3}=2$ and $a_{n-4}=5$. Reading the sequence backwards we get: 1, 1, 2, 1, 5, 4, 11, 1, 32, 49, $\ldots$. This is now sequence A233525 in the OEIS \cite{OEIS}. when read forward, this sequence is a 3-free sequence containing eight divisions in a row. The process can be continued to arbitrarily many terms for arbitrarily many consecutive divisions.

The growth bound for division-rich subsequences is estimated by the following lemma.

\begin{lemma}
For a division-rich subsequence: $\max\{a_{2k+1},a_{2k+2}\} \leq 2\max\{a_{2k-1},a_{2k}\}/3$.
\end{lemma}

\begin{proof}
We can estimate that $a_{2k+1} \leq (a_{2k-1}+a_{2k})/3 \leq 2\max\{a_{2k-1},a_{2k}\}/3$ and $a_{2k+2} \leq (a_{2k}+a_{2k+1})/3 \leq 5\max\{a_{2k-1},a_{2k}\}/9 < 2\max\{a_{2k-1},a_{2k}\}/3$. 
\end{proof}

So we can expect that with probability $1/2^n$ there would be a subsequence of length $2n$, where the maximum of the next two terms does not exceed the maximum of the previous two terms by $2/3$. Clearly it cannot go down forever. We need to start with very large numbers to get a long stretch of a division-rich subsequence.

\begin{lemma}\label{thm:poor}
There exist arbitrary long division-poor subsequences.
\end{lemma}

This proof is more complicated than the previous one, so we will do it together with a proof of a more powerful theorem in next Section~\ref{sec:customized}.

\begin{lemma}
If we index a division-poor subsequence in such a way that division happens on the odd term, then all the even terms form an increasing subsequence: $a_{2k} > a_{2k-2}$.
\end{lemma}

\begin{proof}
As every even term is the sum of the previous two terms we get: $a_{2k} = a_{2k-1} + a_{2k-2} > a_{2k-2}$.
\end{proof}

That means that both division-rich and division-poor subsequences can not form a cycle. In particular, it means we can have sequences of arbitrary length without entering a cycle.

We showed that there exist 3-free Fibonacci sequences that have long increasing subsequences. Still, we want to present a probabilistic argument that any 3-free Fibonacci sequence ends in a cycle.

According to our probabilistic assumptions we divide either every term or every other term with the same probability $1/2$. So on average we divide on every $1.5$ step. But for how much we divide on average? 

\begin{lemma}\label{thm:average3}
On average we divide by $3^{3/2}$.
\end{lemma}

\begin{proof}[Probabilistic Argument]
Now we want to use the fact that when we divide by a power of 3 we on average divide by more than 3. If the number is large, we divide by 3 with probability $2/3$, by 9 with probability $2/9$ and so on. So the average division is by 
$$3^{2/3} \cdot 9^{2/9} \cdot 27^{2/27} \cdot \ldots.$$
We can say it differently. 
We can say that we divide by 3 with probability 1 and additionally we divide by 3 more with probability $1/3$, and by 3 more with probability $1/9$, so the result is 3 to the power
$$1+1/3+1/9+1/27 + 1/81 + \ldots = 3/2.$$
Since the above sum is equal to $3/2$, every time we divide, we on average divide by $3^{3/2}$.
\end{proof}

Notice that the average number we divide by is approximately 5.2 which is more than 5.

Let us build a probabilistic sequence that capture some of the behavior of 3-free Fibonacci sequences. We start with two numbers $a_1$ and $a_2$ and flip a coin. If the coin turns heads we add the next number $a_3 = (a_1+a_2)/5$ to the sequence. If the coin turns tails we add two more terms: $a_3 = a_1+a_2$ and $a_4 = (a_2+a_3)/5$ to the sequence. Then repeat. We expect that this sequence on average grows faster than 3-free Fibonacci sequences, because we divide by a smaller number.

Now we want to bound the maximum of the last two terms of this probabilistic sequence after two coin flips. Let $M = \max\{a_1,a_2\}$. We have the following cases:

\begin{itemize}
\item After two heads, the sequence becomes: $a_1$, $a_2$, $(a_1+a_2)/5$, $(a_1+6a_2)/25$. The last two terms do not exceed $2M/5$.
\item After head, tail, the sequence becomes: $a_1$, $a_2$, $(a_1+a_2)/5$, $(a_1+6a_2)/5$, $(2a_1+7a_2)/25$. The last two terms do not exceed $7M/5$.
\item After tail, head, the sequence becomes: $a_1$, $a_2$, $a_1+a_2$, $(a_1+2a_2)/5$, $(6a_1+7a_2)/25$. The last two terms do not exceed $3M/5$.
\item After two tails the sequence becomes: $a_1$, $a_2$, $a_1+a_2$, $(a_1+2a_2)/5$, $(6a_1+7a_2)/5$, $(7a_1+9a_2)/25$. The last two terms do not exceed $13M/5$.
\end{itemize}

As each event happens with the same probability $1/4$, the average growth after $2n$ coin flips is $(2\cdot 3 \cdot 7 \cdot 13)^{1/4}/5$, which is below 0.97. So the overall trend for this sequence is to decrease.

Based on our computational experiments and probabilistic discussions above we conjecture: 

\begin{conjecture}
Any 3-free Fibonacci sequence ends in a cycle.
\end{conjecture}

So 2-free Fibonacci sequences provably end in a cycle, 3-free sequences conjecturally end in a cycle. Will 4-free Fibonacci sequences end in cycles too? Before discussing 4-free Fibonacci sequences, we want to make a detour and prove the promised result that an arbitrarily long division-poor sequence exists (See Lemma~\ref{thm:poor}) as a corollary to a much stronger and a more general theorem.

\section{Customized-division subsequences}\label{sec:customized}

We promised to give a proof that there exist arbitrary long division-poor sequences that are 3-free Fibonacci sequences. Now we want to prove a stronger statement. We want to allow any $n$ and show that we can build a customized $n$-free Fibonacci sequence that will have a division by a prescribed power of $n$ with the prescribed remainder.

Let us correspond to any $n$-free Fibonacci sequence the list of numbers by which we divide at every step. We call this list \textit{a signature}. For example, a 3-free sequence 5, 4, 1, 5, 2, 7, 1, has signature *, *, 9, 1, 3, 1, 9. We placed stars at the first two places, because we do not know the preceding members of the sequence and, hence, do not know the powers.

Given an $n$-free Fibonacci sequence with a given signature and a given set of remainders, we can build many other sequences with the same signature and the set of remainders.

\begin{lemma}\label{thm:adjustement}
Suppose an $n$-free Fibonacci sequence $s_1$ starts with $a_1$ and $a_2$ that are not divisible by $n$ and the product of the numbers that we divide by while calculating the first $k$ terms is strictly less than $n^m$. Consider an $n$-free Fibonacci sequence $s_2$ that starts with $b_1=a_1+ d_1 n^m$ and $b_2=a_2+ d_2 n^m$ for any integers $d_1$ and $d_2$. The first $k$ terms of both sequences have the same signature and the same set of remainders modulo $n$.
\end{lemma}

\begin{proof}
The initial terms of both sequences have the same remainders modulo $n^m$. Hence their sums have the same remainders. The first time we need to divide, we divide by the same power of $n$, say $m_1$, and the result will have the same remainders modulo $n^{m-m_1}$. The next time we divide, the result will have the same remainders modulo $n^{m-m_1-m_2}$. And so on. When we complete the subsequence, all the remainders will be the same modulo $n$.
\end{proof}

This lemma allows us to find a positive sequence with a given signature if we already found a sequence that might not be all positive.

\begin{cor}
If there exists some finite $n$-free Fibonacci sequence with a given signature and a set of remainders, then there exists a sequence with the same signature and remainders such that every term is positive.
\end{cor}

\begin{proof}
Adjust the initial terms according to Lemma~\ref{thm:adjustement}.
\end{proof}

We say that a finite sequence of remainders $r_i$ modulo $n$ and a finite signature of the same length \textit{match} each other, if the signature has a positive power of $n$ in place $k$ if and only $r_{k-2} + r_{k-1} | n$. We call a sequence of remainders modulo $n$ \textit{legal} if $r_{k-2} + r_{k-1} = r_k$ unless $r_{k-2} + r_{k-1} | n$. Non-legal sequences can not be sequences of remainders of an $n$-free Fibonacci sequence.

\begin{theorem}
Given a legal finite sequence of remainders and a matching signature, there exists an $n$-free Fibonacci sequence with the given sequence of remainders and signature.
\end{theorem}

\begin{proof}
By Lemma~\ref{thm:adjustement} it is enough to find such a sequence that does not need to have positive terms. We can produce such a sequence by building it backwards, similar to what we did in Lemma~\ref{thm:rich}.
\end{proof}

\begin{cor}
There exist arbitrary long division-poor 3-free Fibonacci subsequences.
\end{cor}

Let us build an example of a division-poor 3-free sequence, where each division is by exactly 3. Begin with $a_n=1$, $a_{n-1}=1$, and build a sequence that is not necessarily all-positive. That means we will have $a_{n-2k}=3a_{n-2k+2}-a_{n-2k+1}$ and $a_{n-2k-1}=3a_{n-2k+1}-a_{n-2k}$. Here are several terms: $-8$, 7, $-1$, 2, 1, 1. For a positive version we need to add $3^3$ to $-8$, as outlined in Lemma~\ref{thm:adjustement} as we had two divisions by 3. The adjusted all-positive division-poor sequence is 19, 7, 26, 11, 37, 16.

\section{4-free Fibonacci sequences}\label{sec:4free}

Consider the 4-free Fibonacci sequence starting with 0, 1. This sequence is A224382: 0, 1, 1, 2, 3, 5, 2, 7, 9, 1, 10, 11, 21, 2, 23, 25, $\ldots$. It seems that this sequence grows and does not cycle.

In checking many other 4-free Fibonacci sequences, we still did not find any cycles. The behavior of 4-free sequences is completely different from the behavior of 3-free sequences. 

For 3-free sequences we expected that all of them cycle. Here it might be possible that none of them cycles.

Before making any claims, let us see how these sequences behave modulo 4.

\begin{lemma}
A 4-free Fibonacci sequence contains an odd number.
\end{lemma}

\begin{proof}
Suppose there exists a 4-free Fibonacci sequence containing only even numbers. Then ignoring the initial terms, all the elements of the sequence equal 2 modulo 2. Therefore, we divide by a power of 4 every time. This cannot last forever.
\end{proof}

\begin{lemma}
After the first occurrence of an odd number, a 4-free Fibonacci sequence cannot have two even numbers in a row.
\end{lemma}

\begin{proof}
Start with the first odd number. The steps that do not include division generate a parity pattern: odd, odd, even, odd, odd, even and so on. So there are no two even numbers in a row. That means we can get a multiple of 4 only after summing two odd numbers. We might get an even number after the division, but the next number must be odd again.
\end{proof}

Let us analyze the 4-free Fibonacci sequences probabilistically, similar to what we did for $n=3$.

\begin{lemma}\label{thm:average4}
An average division is by a factor of $4^{4/3} \approx 6.35$.
\end{lemma}

\begin{proof}
When we divide, we divide by 4 with probability 1, additionally with  probability $1/4$ we divide by 4 more, and so on. So the result is 4 to the power
$$1+1/4+1/16+\ldots = 4/3.$$
\end{proof}

How often on average do we divide? Let us assume that we passed stretches of all even numbers. Each time we divide after that, the previous two numbers are odd. After the division, the remainder is 1, 2, or 3. So the following six cases describe what happens after the division:

\begin{lemma}
If $a \equiv 1$ (mod 4) and $b \equiv 3$ (mod 4), then the following set of remainders might happen until the next division:
\begin{itemize}
\item 1;
\item 2, 1, 3;
\item 3, 2, 1, 3.
\end{itemize}
If $a \equiv 3$ (mod 4) and $b \equiv 1$ (mod 4), then the following set of remainders might happen until the next division:
\begin{itemize}
\item 3;
\item 2, 3, 1;
\item 1, 2, 3, 1.
\end{itemize}
\end{lemma}

\begin{lemma}\label{thm:averagesteps}
The average number of steps between divisions is $8/3$.
\end{lemma}

\begin{proof}[Probabilistic argument]
We assume that during the division each remainder is generated with probability $1/3$, so the stretches of length 1, 3, and 4 between divisions are equally probable.
\end{proof}

We want to build a probabilistic model that reflects some behavior of 4-free Fibonacci sequences. Let us denote the average factor by which we divide by $x$. In this model, we simply divide by $x$ each time we need to divide. The sequence stops being an integer sequence, but we artificially assign a remainder mod 4 to every element of the sequence to see when we need to divide. We want to show that our model sequence grows with probability 1.

First, we want to estimate the ratio of two consecutive numbers in the model sequence.

\begin{lemma}\label{thm:ratio}
If $a$ and $b$ are two consecutive numbers in the model sequence, starting from index 3, then $b > a\frac{2+x}{(1+x)x}$ and $a > b\frac{1}{1+x}$.
\end{lemma}

\begin{proof}
Let $v$ be the element before $a$. Then $b \geq (a+v)/x > a/x$. Analogously, $b \leq a+v$, and by the previous sentence, $xa > v$. Therefore, $b < a(1+x)$, which means that $a/(1+x) < v$. Plugging this back into $b \geq (a+v)/x$, we get $b > a(2+x)/(1+x)x$.
\end{proof}

Now we are ready to prove the theorem:

\begin{theorem}
In our probabilistic model, a sequence grows with probability 1.
\end{theorem}

\begin{proof}
Suppose we have two consecutive numbers in the sequence $a$ and $b$ whose sum is divisible by 4. By Lemma~\ref{thm:averagesteps} we have 1, 3, or 4 terms until the following division with the same probability. That is, the following continuations until the next division are equally probable:
\begin{itemize}
\item $a$, $b$, $(a+b)/x$.
\item $a$, $b$, $(a+b)/x$, $(a+(x+1)b)/x$, $(2a+(x+2)b)/x$.
\item $a$, $b$, $(a+b)/x$, $(a+(x+1)b)/x$, $(2a+(x+2)b)/x$, $(3a+(2x+3)b)/x$.
\end{itemize}

If $b>a$, then the maximum of the last two terms is: $b$, $(2a+(x+2)b)/x$, and $(3a+(2x+3)b)/x$ correspondingly. Adjusting for the fact that $a > b\frac{1}{1+x}$ (see  Lemma~\ref{thm:ratio}), the maximum of the last two terms is at least $b$, $b((x+2)+2/(x+1))/x$, and $b((2x+3)+3/(x+1))/x$ correspondingly. We want to estimate the ratio of the maximum of the last two terms to $\max\{a,b\}$. Counting probabilities we get the following lower bound for the ratio
$$1^{1/3}(((x+2)+2/(x+1))/x)^{1/3}(((2x+3)+3/(x+1))/x)^{1/3}= 1.51023.$$

If $b \leq a$, then the maximum of the last two terms is at least $(a+b)/x$, $(2a+(x+2)b)/x$, or $(3a+(2x+3)b)/x$, respectively. Using $b > a\frac{2+x}{(1+x)x}$ from Lemma~\ref{thm:ratio}, we get that the maximum of the last two terms is at least $a\frac{2+2x+x^2}{(1+x)x^2}$, $a\frac{2x^2+6x+4}{(1+x)x^2}$, or $a\frac{5x^2+10x+6}{(1+x)x^2}$, respectively. Counting probabilities, we get that the maximum is multiplied by at least 
$$\frac{2+2x+x^2}{(1+x)x^2}^{1/3}\frac{3x^2+6x+4}{(1+x)x^2}^{1/3}\frac{5x^2+10x+6}{(1+x)x^2}^{1/3}=0.453822.$$

Notice that if we have 3 or 4 terms until the next division, then the last two terms before the division are in the increasing order. That means the case when $b \leq a$ is at least twice less probable, so the average growth is at least the cube root of $1.51023^2\cdot 0.453822 = 1.03507$, which is greater than 1.
\end{proof}

The result is greater than 1, which means that our model sequence does not cycle all the time. Therefore, extending the argument to 4-free sequences, we can safely say that 4-free sequences do not cycle all the time. Taking our computational experiments and our intuition into account we are comfortable with the following conjecture:

\begin{conjecture}
With probability 1 a 4-free Fibonacci sequence does not cycle.
\end{conjecture}

So 4-free Fibonacci sequences do not cycle. If $n$ grows will it mean that $n$-free Fibonacci sequences for $n > 4$ will not cycle either?

\section{5-free Fibonacci sequences}\label{sec:5free}

Let us look at the Lucas sequence modulo 5: $2$, $1$, $3$, $4$, $2$, 1 $\ldots$ and see that no term is divisible by 5. Clearly, no term in the Lucas sequence will require that we factor out a power of 5, and the terms will grow indefinitely. Thus, the Lucas sequence is itself a 5-free Fibonacci sequence. This is something new. We do not need a probabilistic argument to show that there are 5-free Fibonacci sequences that do not cycle.

On the other hand, it becomes quickly evident that the sequence of 5-free Fibonacci numbers: 0, 1. 1, 2, 3, 1, 4, 1, 1, 2, $\ldots$ (see A214684) cycles. Some sequences cycle, and some clearly do not!

But how often will we come upon a sequence that grows indefinitely? To answer this question, let us look at a couple of terms from a few sequences of Fibonacci numbers modulo 5. Begin with $1, 1, \ldots$, to obtain the sequence

\vspace{3pt}
\centerline{1, 1, 2, 3, 0,}
\centerline{3, 3, 1, 4, 0,}
\centerline{4, 4, 3, 2, 0,}
\centerline{2, 2, 4, 1, 0,}
\vspace{3pt}

and so on. We write it like this for clarity: at the end of each line, the last term is divisible by 5. In particular, the table above shows that $Z(5)$---the entry point of 5---is 5. Furthermore, since 5 is prime, we could know beforehand that each of these lines would be the same length. We simply had to start with the line beginning $1, 1, \ldots$, and multiply each term by 2, then 3, then 4. Clearly no term in the line could become 0 after the multiplication (except, of course, 0 itself), since there are no two non-zero numbers that multiply to zero in a field. Buy the way, the sequence of Fibonacci entry points for primes is A001602.

There were exactly $5-1=4$ lines, and 5 elements (and therefore, 5 consecutive-element pairs) in each line, for a total of 20 pairs. But, excluding (0, 0), there are $5^2-1=24$ possible pairs! Thus, our 4 extra pairs must have gone into another cycle. This cycle could not contain any multiple of 5, and therefore serves as a witness that non-cyclic sequences exist in some 5-free Fibonacci sequence. There many books and paper about Fibonacci numbers and Lucas numbers, we mostly used \cite{HW,Va89} for reference. This argument shows what we already know: that a Fibonacci-like sequence that does not contain multiple of 5 exists. 

We see that there is a strong connection between free Fibonacci sequences and proper Fibonacci sequences. Maybe we can study entry points of Fibonacci sequences to see if the story with 5 repeats for some other numbers.

\section{Division-Free $n$-Free Fibonacci Sequences}\label{sec:division-free}

Let us call an integer $n$ a \textit{Fibonacci omni-factor} if any Fibonacci-like sequence contains a multiple of $n$. We just saw that 5 is the smallest integer that is not a Fibonacci omni-factor.

If a number $n$ is not a Fibonacci omni-factor, then there exists a Fibonacci-like sequence that is at the same time an $n$-free Fibonacci sequence. 

Prime omni-factors can be found with the help of the following well known lemma.

\begin{lemma}
A prime $p$ is a Fibonacci omni-factor if and only if $Z(p) = p+1$.
\end{lemma} 

\begin{proof}
Consider a section of the Fibonacci sequence modulo $p$ before the entry point. Multiply this subsequence by any other remainder modulo $p$. A prime $p$ is not an omni-factor if there are fewer than $p^2-1$ total elements in all the lines of the Fibonacci sequences beginning with $k, k, \ldots$ modulo $p$. Then, as all lines are the same length as that beginning with 1, 1, $\ldots$ (that is, the start of the Fibonacci sequence proper), it will suffice to show that $(p-1)\cdot Z(p) < p^2-1$, or $Z(p) < p+1$.
\end{proof}

It is clear that $Z(p)(p-1) + 1$ can not be more than $p^2$. Hence we just proved a well known fact:

\begin{cor}
For every prime $p$, $Z(p) \leq p+1$. 
\end{cor} 

Examples of primes that are not omni-factors include 5, which divides $F_5$, 11, which divides $F_{10}$, and 13, which divides $F_7$. The corresponding sequence is now sequence A230359. Omni-factor primes are the primes $p$ such that any Fibonacci-like sequences contains multiples of $p$. This sequence is A000057. By the way it is not known whether the latter sequence is infinite \cite{CR}.

The entry points can be defined not only for primes, see sequence A001177 of Fibonacci entry points. For a composite number $n$ the relationship between the entry point $Z(n)$ and the existence of Fibonacci-like sequences not divisible by $n$ is slightly more complicated. But it is possible to check computationally if the sequences that start with zero and another number contain all possible pairs of remainders. The sequence of Fibonacci omni-factors, that is of numbers $n$ such that any Fibonacci-like sequences contains multiples of $n$ is A064414. Correspondingly the numbers $n$ such that there exist a Fibonacci-like sequence without multiples if $n$ is the complement of A064414. It is now sequence A230457. It starts as 5, 8, 10, 11, 12, 13, 15. The first composite number in the sequence is 8. 

Lucas numbers again is a sequence that does not have multiples of 8. Lucas numbers provide examples for many numbers, namely for the numbers that they do not divide. These numbers are represented by the sequence A064362. Thus, Lucas numbers provide examples for 10, 12, 13, 15 and so on in addition to 5 and 8.

The smallest number that Lucas numbers does not provide an example for is 11. For 11, we can start with 1 and 4, to get the sequence A000285: 1, 4, 5, 9, 14, 23, 37 and so on. This is a Fibonacci-like sequence that is never divisible by 11.

All non-omni-factors that are not factors of the Lucas numbers are: 11, 18, 19, 22, 29, 31, 38, 41, 44, 46, 47, 54, 58, 59, 62, 71, 76, 79, 82, 94 and so on. This sequence is the A230457 sequence intersecting with factors of Lucas numbers: sequence A065156. It is now sequence A232658. Given that A064362 (Numbers $n$ such that no Lucas number is a multiple of $n$) is the complement of A065156, the new sequence can be defined as the sequence A230457 from which the numbers from A064362 are removed.

Here we found many sequences that are Fibonacci-like and are not divisible by some number $n$. They provide an example of infinitely growing $n$-free sequences. Moreover, the division never happens. 

Will we ever see more cycles?

\section{Other Cycles}\label{sec:cycles}

For $n=2$  every sequence ends in a cycle of length 1. For $n=3$ every sequence we checked ended with a cycle of length 3. For $n=4$ we did not find any cycles at all. 

So far we found an example of a 5-free Fibonacci sequence that need not ever be divided by 5. We also found a cycle: 1, 1, 2, 3, 1, 4. Are there other cycles? We know that we can multiply a 5-free Fibonacci sequence by any number that is not divisible by 5 to get another 5-free Fibonacci sequence. Thus, we have many other cycles among 5-free Fibonacci sequences: 2, 2, 4, 6, 2, 8, and 3, 3, 6, 9, 3, 12, and so on.

In general, we can multiply an $n$-free sequence by a number coprime with $n$ to get another sequence. Also, a more subtle statement is true. if we multiply an $n$-free sequence by a number not necessary coprime with $n$ but the result does not contain multiples of $n$, then the result of the multiplication is an $n$-free sequence.

If all the elements of an $n$-free sequence are divisible by a number $m$, we can divide the sequence by $m$ to get another $n$-free sequence. We would like to point out that $m$ does not need to be coprime with $n$. This warrants a definition. Call a cycle \textit{primitive} if its terms are coprime. As we just explained the following lemma is true.

\begin{lemma}
Any cycle can be divided by an integer to become a primitive cycle.
\end{lemma}

All the primitive cycles we found so far contained only numbers below $n$. There is no reason why this property should hold for any cycle. For example, here is another primitive cycle of 5-free Fibonacci sequences: 4, 3, 7, 2, 9, 11, 4, 3.

Though we did not find any more cycles, in case they exist, we can prove some of their properties.

Take, for example, Lemma~\ref{thm:parity} where we used parity to prove that any 3-free Fibonacci cycle has length that is a multiple of 3. We can replace 3 by any odd number in Lemma~\ref{thm:parity} to get the following lemma.

\begin{lemma}
The length of an $n$-free Fibonacci cycle is divisible by 3 for odd $n$.
\end{lemma} 

It is not surprising that all the 5-free Fibonacci cycles that we found so far has length 6.

There is no reason we should restrict ourselves with parity: remainders of Fibonacci numbers modulo 2. The Fibonacci sequence modulo $n$ is periodic. The period is called \textit{Pisano period} and is denoted as $\pi(n)$. The sequence of Pisano periods is A001175. Pisano periods for prime numbers are outsourced into A060305.

For example Fibonacci numbers modulo 3 form a cycle of length 8: 0, 1, 1, 2, 0, 2, 2, 1. We can generalize Lemma~\ref{thm:parity} to any prime number.

\begin{lemma}
The length of an $n$-free Fibonacci cycle is divisible by $\pi(p)$, where $p$ is a prime factor of $n-1$.
\end{lemma} 

\begin{proof}
Any cycle has the same length as a primitive cycle. Dividing by $n$ or its power does not change a remainder modulo $p$ as $n \equiv 1$ (mod $p$). All non-trivial Fibonacci cycles modulo $p$ are of the same length $\pi(p)$. So the primitive $n$-free Fibonacci cycle modulo $p$ has to be a multiple of $\pi(p)$.
\end{proof}

For example, any 4-free Fibonacci cycle, if it exists, is of length $8k$ for some $k$. This is due to the fact that $4 \equiv 1$ (mod $3$) and $\pi(3) = 8$.

The following theorem immediately follows.

\begin{theorem}
Let $p_i$ be prime factors of $n-1$. Then the cycles in $n$-free Fibonacci sequences are of a length divisible by $\gcd(\pi(p_i))$.
\end{theorem}

We studied cycles, but we actually do not expect many of them, as we expect the $n$-free Fibonacci sequences to grow faster for larger $n$.

As the number $n$ grows, the multiples of $n$ are more spread apart in the Fibonacci sequence, that means the division happens more rarely. We think that the increase in the number by which we divide is less pronounced than the fact that the divisions are more spread apart.

In the next Section we look at computational results and make probabilistic arguments to show that for $n > 3$ cycles should appear very rarely.

\section{Growth in Division-Free Sequences}\label{sec:growthdivfree}

We ran several experiments. In our first experiment we did the following:

For each $n$ we built 10000 random $n$-free Fibonacci sequences of length 500. Namely, we picked initial terms of each sequence as two random numbers between 1 and 1000. Then we averaged each term and found the best approximation for the exponential growth. We did this 5 times to confirm consistency of the exponents. That is, we approximated the $m$-th term of the $n$-free Fibonacci sequence as $g(n)^m$, where $g(n)$ is described by the following sequence starting from $n=4$: 1.32, 1.61, 1.42, 1.34, 1.61, 1.4, 1.61, 1.61, and so on. We did this for $n$ up to 50.

Our experimental results showed that for $n$ for which division-free $n$-free Fibonacci sequences exist the growth is the same and it is about 1.61. Can we explain this? Let us take a closer look at the smallest such $n$: 5.

Consider an arbitrary 5-free Fibonacci sequence. When we divide by a power of 5 at some point we may crossover to a division-free sequence. If we ever get two consecutive remainders as in the Lucas numbers: 2, 1, 3, 4, 2, and so on, we will never divide by a power of 5 again. Notice that the Lucas numbers modulo 5 cycle. The cycle has length 4 and contain every remainder exactly once. 

That means if the number in the sequence before division has remainder $r$, then we crossover into a division free sequence when the next number is congruent to exactly one out of possible four remainders modulo 5: 1, 2, 3, or 4. Consider for example the sequence starting with 1, 6. It continues to 7, 13, 4, and from here we would never divide by 5 again. 

Assume that the remainder after the division is chosen randomly with a uniform distribution. In this case, there is a 25\% chance of entering the cycle with no multiples of 5, and a 75\% chance of entering a sequence which will be divided by 5 again. 

Unless we enter into a cycle, as the number of these randomizations increases, it becomes more likely that the sequence will crossover into a division-free sequence.

We did not find many cycles. Moreover all primitive cycles that we found had small numbers in it. Suppose that there are no primitive cycles with large numbers. Then, if we start with two large coprime numbers, there would be many potential divisions on the way to a cycle. Therefore, the probability of entering a division-free sequence will be very large. This probabilistic argument leads us to a conjecture.

\begin{conjecture}
If we pick the starting integers in the range from 1 to $N$ the probability that we end up in a division-free 5-free Fibonacci sequence tends to 1 when $N$ tends to $\infty$.
\end{conjecture}

Let us remind you that division-free sequences are Fibonacci-like sequences. They grow like $\phi^n$, where $\phi$ is the golden ratio. It is not surprising that we get 1.61 as the growth ratio: the number close to the golden ratio, but slightly below it.

We explained why 1.61 is the exponent for $n=5$. What about other numbers?

Let us start by looking at the proportion of pairs that generate sequences not containing zeros. We submitted two new sequences to the OEIS:

\begin{itemize}
\item A232656 The number of pairs of numbers below $n$ that, when generating a Fibonacci-like sequence modulo $n$, contain zeros: 1, 4, 9, 16, 21, 36, 49, 40, 81, $\ldots$.
\item A232357 The number of pairs of numbers below $n$ that, when generating a Fibonacci-like sequence modulo $n$, do not contain zero: 0, 0, 0, 0, 4, 0, 0, 24, 0, $\ldots$.
\end{itemize}

The sum of the two sequences is the sequence of squares: $a(n) = n^2$: the total number of possible pairs of remainders modulo $n$. For our argument we are interested in the ratio A232357$(n)/(n-1)^2$: the proportion of pairs not containing zero that lead to division-free sequences. This is what we get starting from $n=2$:

\begin{quote}
0, 0, 0.25, 0, 0, 0.49, 0, 0.20, 0.20, 0.40, 0.58, 0, 0.18, 0.53, 0.56, 0.50, 0.11, 0.18, 0.72, 0.18, 0, 0.68, 0.18, 0.54, 0, 0.40, 0.57, 0.17, 0.067, 0.52, 0.57, 0.79, 0.17, 0.74, 0.53, 0.58, 0.52, 0.50, 0.55, 0.69, 0, 0.50, 0.17, 0.52, 0.70, 0.81, 0, 0.17, 0.52, 0.52, 0.52, 0.51, 0.86, 0.67, 0.52, 0.55, 0.034, 0.46, 0.78, 0.55, 0.75
\end{quote}

There is no clear pattern. For example, it drops significantly for $n=59$. Let us not get upset yet. If we look at the exact number A232357$(59) = 116 = 2 \cdot 58$. Is it true that every time we divide the probability of getting into a division-free sequence is the same?

Suppose the term of the $n$-free Fibonacci sequence before division has remainder $r$. The probability that we crossover to a division-free sequence is $f(r)/(n-1)$, where $f(r)$ is the number of possible remainders that follow $r$ in division-free sequences.

\begin{lemma}
If $n$ is prime, then $f(r)$ is constant, for $r \neq 0$.
\end{lemma}

\begin{proof}
Consider the Fibonacci sequence modulo $n$ between 0 and the first entry point. Multiply this sequence by all numbers below $n$. Pick the set of all possible pairs we get. Every non-zero number will be in this set of pairs the same number of times. But these are the pairs that lead to divisions. That means that in the set of pairs that lead to division-free sequences each remainder over zero is contained there the same number of times. As each remainder can not be followed by the same number of times, $f(r)$ must be constant.
\end{proof}

The previous lemma shows that the argument we provided for 5 works for every prime number that is not a Fibonacci omni-factor. Each time we divide, we crossover into a division-free sequence with the same probability. 

But what about composite numbers that are not omni-factors? We know that the cycles of Fibonacci-like sequences modulo a composite number might be different length. The number of different cycles modulo $n$ is A015134. Correspondingly, A015135 is the number of different period lengths.

Let us look at the smallest possible case of a composite non omni-factor: $n=8$.

Let us check all possible Fibonacci-like sequences modulo 8. There is a cycle 0, 1, 1, 2, 3, 5, 0. This cycle is of length 6. If we multiply it by 2, 3, 5, 6, or 7 we get 5 more cycles of length 6. There is a cycle 0, 4, 4, 0 of length 3. There is a trivial cycle 0, 0 of length 1.

We also have a cycle corresponding to Lucas numbers 1, 3, 4, 7, 3, 2, 5, 7, 4, 3, 7, 2, 1, 3, $\ldots$ of length 12. There is also another cycle, that is a multiple of this cycle: 1, 4, 5, 1, 6, 7, 5, 4, 1, 5, 6, 3, 1, 4, $\ldots$. These are two cycles of length 12. There is no way each of the 7 possible non-zero remainders participates in these cycles the same number of times.

Table~\ref{tbl:divfreepairs} shows, given a remainder, what the next remainder should be if we are inside a division-free sequence. 

\begin{table}[htbp]
\begin{center}
\begin{tabular}{| c | c |}
\hline
1 & 3, 4, 5, 6 \\
2 & 1, 5 \\
3 & 1, 2, 4, 7 \\
4 & 1, 3, 5, 7 \\
5 & 1, 4, 6, 7 \\
6 & 3, 7\\
7 & 2, 3, 4, 5\\
\hline
\end{tabular}
\caption{Remainders in division-free pairs.}\label{tbl:divfreepairs}
\end{center}
\end{table}

We see that numbers 1, 3, 4, 5, 7 correspond to 4 possibilities, and numbers 2 and 6 to two possibilities.

That means the probability that we crossover to a division-free sequence depends on the previous remainder. But the important part is that it is never zero, because each remainder has at list two numbers that follow it.

We can assume that with each division we crossover to a division-free sequence with some probability that is bounded from below. From this assumption we get a conjecture. 

\begin{conjecture}
For any $n$ such that a division-free sequence exists, if we pick the starting integers in the range from 1 to $N$ the probability that we eventually end up in a division-free $n$-free Fibonacci sequence tends to 1 when $N$ tends to $\infty$.
\end{conjecture}

This conjecture explains why we get 1.61 as a growth estimate for non-omni-factors in our data. Now we want to explain other numbers in the next section.

\section{Growth rates for omni-factors}\label{sec:growthomni}

Table~\ref{tbl:exp} shows the exponents that are not equal to 1.61. We call them the deviated exponents and they appear when $n$ is an omni-factor.

\begin{table}[htbp]
\begin{center}
\begin{tabular}{| l | r rrrrrrrr|}
\hline
$n$ & 4 & 6 & 7 & 9 & 14 & 23 & 27 & 43 & 49\\
growth & 1.32 & 1.42 & 1.34 & 1.4 & 1.49 & 1.48 & 1.53 & 1.54 & 1.56\\
\hline
\end{tabular}
\caption{Deviated exponents.}\label{tbl:exp}
\end{center}
\end{table}

It is not surprising that these numbers are smaller than the golden ratio. Indeed, in every sequence we divide by a power of $n$ an infinite number of times.

Moreover, we can extend the reasoning from Lemma~\ref{thm:average3} and Lemma~\ref{thm:average4} to see that each time we divide, we divide on average by $n^{n/(n-1)}$. The number we divide by grows with $n$, but the numbers in the second row of Table~\ref{tbl:exp} do not decrease. To explain this we need to see how often we divide. If we start with a pair of remainders, what is the average number of steps we need to make to get to zero? If we take all the pairs from the same cycle, then the average would be about half of the length of the cycle.

The total average is the sum of squares of cycle lengths divided by the total number of pairs and by 2. This is not an integer sequence, but we submitted the rounded number and it is now sequence A233248. To keep the memory of the precise number we submitted the sum of the squares of cycle lengths as sequence A233246.

We can approximate the average number of steps to the next division as $Z(n)/2$, but this is not precise. For example, we saw before that for 4-free numbers the number of steps until the next division is 1, 3, or 4. So the average is 8/3. It is easy to calculate this number when $n$ is prime. After the division we get a number that is not divisible by $n$, and the previous number is not divisible by $n$. We can assume that all such pairs are equally probable and the average number of steps is $(Z(n)-1)/2$. For non-prime numbers we can calculate this explicitly keeping in mind that before the division both numbers are coprime with $n$. (We can probabilistically argue that eventually elements of a sequence become coprime.)

If $a$ is the average number of steps until the next division, then the estimated average division is by $n^{2n/(n-1)a}$.

We combined these numbers in Table~\ref{tbl:growthexplanation}. The third row is the entry points, the fourth row is the average number of steps until the next division. The fifth row is the calculated average division per step. 

The last row in the table needs a separate explanation. Suppose $d$ is the average division per step. We took a recurrence defined as: $x_n = (x_{n-1} + x_{n-2})/d$ and calculated its growth, which is the last row.

\begin{table}[htbp]
\begin{center}
\begin{tabular}{| l | r rrrrrrrr|}
\hline
$n$ & 4 & 6 & 7 & 9 & 14 & 23 & 27 & 43 & 49\\
experimental growth & 1.32 & 1.42 & 1.34 & 1.4 & 1.49 & 1.48 & 1.53 & 1.54 & 1.56\\
$Z(n)$ & 6 & 12 & 8 & 12 & 24 & 24 & 36 & 44 & 56\\
$a$ & 8/3 & 6 & 7/2 & 45/8 & 154/13 & 23/2 & 459/26 & 43/2 & 441/16\\
average division & 2.00 & 1.43 & 1.91 & 1.55 & 1.27 & 1.33 & 1.21 & 1.20 & 1.16\\
recurrence growth & 1 & 1.26 & 1.03 & 1.19 & 1.36 & 1.32 & 1.41 & 1.42 & 1.46\\
\hline
\end{tabular}
\caption{Estimated growth.}\label{tbl:growthexplanation}
\end{center}
\end{table}

The last line, the recurrence growth and the fourth line---the average number of steps until division---are strongly correlated with the experimental growth.

\bigskip
\hrule
\bigskip

\noindent 2010 {\it Mathematics Subject Classification}: Primary 11B39; Secondary 11B50.

\noindent \emph{Keywords: } Fibonacci numbers, Lucas numbers, divisibility, entry points.

\bigskip
\hrule
\bigskip

\noindent
(Mentions A000032, A000045, A000057, A000285, A001175, A001177, A015134, A015135, A001602, A060305, A064362, A064414, A065156,  A078414, A214684, A224382. New sequences A230359, A230457, A232656, A232357, A232658, A232666, A233246, A233248, A233525, A233526)

\bigskip
\hrule
\bigskip

\end{document}